\documentclass[11pt]{amsart}

\usepackage[pdftex]{graphicx} \usepackage[dvipsnames]{xcolor}

\usepackage[margin=2cm]{geometry}

\usepackage{amsfonts} \usepackage{amsmath} \usepackage{amssymb}
\usepackage{amsthm} \usepackage{mathtools}
\usepackage{tikz}
\usepackage{tikz-3dplot}

\usepackage{paralist} \usepackage[colorlinks,cite
color=blue,pagebackref=true,pdftex]{hyperref}

\usepackage[T1]{fontenc}

\usepackage{caption} \usepackage{subcaption}

\newcommand\Defn[1]{\textbf{\color{black}#1}}
\newcommand\Def[1]{\Defn{#1}}

\newcommand\ZO{\{0,1\}}
\renewcommand\c{\mathbf{c}}
\renewcommand\u{\mathbf{u}}
\renewcommand\v{\mathbf{v}}

\newcommand\p{\mathbf{p}}
\newcommand\q{\mathbf{q}}

\renewcommand\emptyset{\varnothing}
\newcommand\Z{\mathbb{Z}}

\newcommand\R{\mathbb{R}}

\newcommand\x{\mathbf{x}}
\newcommand\y{\mathbf{y}}

\newcommand\inner[1]{\langle {#1} \rangle}
\newcommand\defeq{\coloneqq}
\newcommand\eqdef{\eqqcolon}

\newcommand\Cube{\Box}

\newcommand\0{\mathbf{0}}
\newcommand\1{\mathbf{1}}
\newcommand\e{\mathbf{e}}

\DeclareMathOperator{\conv}{conv}

\DeclareMathOperator{\vol}{vol}

\newtheorem{thm}{Theorem}[section] \newtheorem{cor}[thm]{Corollary}
 \newtheorem{prop}[thm]{Proposition}

\theoremstyle{definition}

\title{$S$-hypersimplices, pulling triangulations, and monotone paths}

\author{Sebastian Manecke}
\author{Raman Sanyal} 
\author{Jeonghoon So}

\address{Institut f\"ur Mathematik, Goethe-Universit\"at
Frankfurt, Germany} 
\email{sanyal@math.uni-frankfurt.de}
\email{manecke@math.uni-frankfurt.de}
\email{jeonghoon.so@stud.uni-frankfurt.de}


\date{\today}

\begin{document}

\begin{abstract}
    An $S$-hypersimplex for $S \subseteq \{0,1, \dots,d\}$ is the convex hull of
    all $0/1$-vectors of length $d$ with coordinate sum in $S$. 
    These polytopes generalize the classical hypersimplices as well as
    cubes, crosspolytopes, and halfcubes. 
    In this paper we study faces and
    dissections of $S$-hypersimplices. Moreover, we show that monotone path
    polytopes of $S$-hypersimplices yield all types of multipermutahedra. 
    In analogy to cubes, we also show that the number of simplices in a
    pulling triangulation of a halfcube is independent of the pulling order.
\end{abstract}

\maketitle

\section{Introduction}\label{sec:intro}
The \Def{cube} $\Cube_d = [0,1]^d$ together with the \Def{simplex} $\Delta_d =
\conv(\0,\e_1,\dots,\e_d)$ and the \Def{cross-polytope} $\Diamond_d =
\conv(\pm \e_1,\dots,\pm \e_d)$ constitute the \emph{Big Three}, three
infinite families of convex polytopes whose geometric and combinatorial
features make them ubiquitous throughout mathematics.  A close cousin to the
cube is the \Def{(even) halfcube}
\[
    H_d \ \defeq \ \conv \left( \p \in \ZO^d \; : \; p_1 + \cdots + p_d  \text{
    even} \right ) \, .
\]
The halfcubes $H_1$ and $H_2$ are a point and a segment, respectively, but for
$d \ge 3$, $H_d \subset \R^d$ is a full-dimensional polytope. The
$5$-dimensional halfcube was already described by Thomas Gosset~\cite{Gosset} in
his classification of semi-regular polytopes. In contemporary mathematics,
halfcubes appear under the name of \emph{demi(hyper)cubes}~\cite{coxeter} or
\emph{parity polytopes}~\cite{Yannakakis}.  In particular the name `parity
polytope' suggests a connection to combinatorial optimization and polyhedral
combinatorics; see~\cite{CarrKonjevod, ErmelWalter} for more. However,
halfcubes also occur in algebraic/topological
combinatorics~\cite{Green,Green2}, convex algebraic geometry~\cite{SSS}, and
in many more areas.

In this paper, we investigate basic properties of the following class of
polytopes that contains cubes, simplices, cross-polytopes, and halfcubes.  For
a nonempty subset $S$ of $[0,d] \defeq \{0,1,\dots,d\}$, we define the
\Def{$S$-hypersimplex}
\[
    \Delta(d,S) \ \defeq \ \conv \Bigl( \v \in \ZO^d \; : \; v_1 + v_2 + \cdots
    + v_d \in S \Bigl) \, .
\]
In the context of combinatorial optimization these polytopes were studied by
Gr\"otschel~\cite{groetschel} associated to \emph{cardinality homogeneous set
systems}.
Our name and notation derive from the fact that if $S = \{ k \}$ is a singleton, then
$\Delta(d,S) \eqdef \Delta(d,k)$ is the well-known \Def{$(d,k)$-hypersimplex},
the convex hull of all vectors $\v \in \ZO^d$ with exactly $k$ entries equal
to $1$. This is a $(d-1)$-dimensional polytope for $0 < k < d$ that makes
prominent appearances in combinatorial optimization as well as in algebraic
geometry~\cite{MS}. We call $S$ \Def{proper}, if $\Delta(d,S)$ is a
$d$-dimensional polytope, which, for $d > 1$, is precisely the case if $|S|
\neq 1$ and $S \neq \{ 0, d \}$. For appropriate choices of $S \subseteq
[0,d]$, we get
\begin{compactitem}[--]
\item the cube $\Cube_d = \Delta(d, [0,d])$,
\item the even halfcube $H_d = \Delta(d, [0,d] \cap 2\Z)$,
\item the simplex $\Delta_d = \Delta(d, \ZO)$, and
\item the cross-polytope $\Delta(d,\{1,d-1\})$ (up to linear
    isomorphism). 
\end{compactitem}

\newcommand\Mono{\mathcal{M}}%
\newcommand\card{\mathsf{c}}%
In Section~\ref{sec:card}, we study the vertices, edges, and facets of
$S$-hypersimplices.

Our study is guided by a nice decomposition of
$S$-hypersimplices into \emph{Cayley polytopes} of hypersimplices.

In Section~\ref{sec:pull} we return to the halfcube. A combinatorial
$d$-cube has the interesting property that all pulling triangulations have
the same number of $d$-dimensional simplices. The Freudenthal or staircase
triangulation is a pulling triangulation and shows that the number of
simplices is exactly $d!$. We show that the number of simplices in any
pulling triangulation of $H_d$ is independent of the order in which the
vertices are pulled. Moreover, we relate the full-dimensional simplices in any
pulling triangulation of $H_d$ to \emph{partial} permutations and show that
their number is given by
\[
    t(d) \ = \ \sum_{l=3}^{d} \frac{d!}{l!} \left(2^{l-1} - l \right) \, .
\]

For a polytope $P \subset \R^d$ and a linear function $\ell : \R^d \to \R$,
Billera and Sturmfels~\cite{BS} associate the \Def{monotone path polytope}
$\Sigma_\ell(P)$.This is a $(\dim P - 1)$-dimensional polytope whose vertices
parametrize all \emph{coherent} $\ell$-monotone paths of $P$. As a particularly
nice example, they show in~\cite[Example~5.4]{BS} that the monotone path
polytope $\Sigma_\card(\Cube_d)$, where $\card$ is the linear function $\card(\x)
= x_1 + x_2 + \cdots + x_d$, is, up to homothety, the polytope
\[
    \Pi_{d-1} \ = \ \conv( (\sigma(1),\dots,\sigma(d)) \; : \; \sigma \text{
    permutation of } [d]) \, .
\]
For a point $\p \in \R^d$, the convex hull of all permutations of $\p$ is
called the
\Def{permutahedron} $\Pi(\p)$ and we refer to $\Pi_{d-1} = \Pi(1,2,\dots,d)$
as the \Def{standard} permutahedron. If $\p$ has $d$ distinct coordinates,
then $\Pi(\p)$ is combinatorially (even normally) equivalent to $\Pi_{d-1}$.
For the case that $\p$ has repeated entries, these polytopes were studied by
Billera-Sarangarajan~\cite{BilleraSarangarajan} under the name of
\emph{multipermutahedra}. In Section~\ref{sec:mono}, we study maximal
$\card$-monotone paths in the vertex-edge-graph of $\Delta(d,S)$. We show that
all $\card$-monotone paths of $\Delta(d,S)$ are coherent and that essentially
all multipermutahedra $\Pi(\p)$ for $\p \in [0,d-1]^d$ occur as monotone path
polytopes of $S$-hypersimplices. 

We close with some questions and ideas regarding $S$-hypersimplices in 
Section~\ref{sec:misc}.

\textbf{Acknowledgements.} This paper grew out of a project that was part of
the course \emph{Polytopes, Triangulations, and Applications} at Goethe
University Frankfurt in spring 2018. We thank Anastasia Karathanasis for her
support in the early stages of this project. We also thank Jes\'{u}s de Loera,
Georg Loho, and the anonymous referee for many helpful remarks.

\section{$S$-hypersimplices}\label{sec:card}

The vertices of the $d$-cube can be identified with sets $A \subseteq [d]$ and
we write $\e_A \in \ZO^d$ for the point with $(\e_A)_i = 1$ if and only if $i
\in A$. Let $S \subseteq [0,d]$. Since $\Delta(d,S)$ is a vertex-induced
subpolytope of the cube, it is immediate that the vertices of $\Delta(d,S)$
are in bijection to
\[
    \binom{[d]}{S} \ \defeq \ \{ A \subseteq [d] : |A| \in S \} \, .
\]
This gives the number of vertices as $|V(\Delta(d,S))| = \sum_{s \in
S}\binom{d}{s}$. 

For a polytope $P \subset \R^d$ and a vector $\c \in \R^d$, let 
\[
    P^\c \ \defeq \ \{ \x \in P : \inner{\c,\x} \ge \inner{\c,\y} \text{ for
    all } \y \in P \}
\]
be the \Def{face in direction $\c$}.  For example, unless $S = \{0\}$,
$\Delta(d,S)^{\e_i}$ is the convex hull of all $\e_A$ with $A \in
\binom{[d]}{S}$ with $i \in A$. Likewise, unless $S = \{d\}$,
$\Delta(d,S)^{-\e_i} = \conv(\e_A : A \in A \in \binom{[d]}{S},i \not\in A)$.
Under the identification $\{ \x : x_i = 1 \} \cong \R^{d-1}$, this gives for
$|S| > 1$
\begin{equation}\label{eqn:Delta_facets}
    \begin{aligned}
    \Delta(d, S)^{\e_i} \ &\cong \ \Delta(d - 1, S^+)
    \qquad \text{where } S^+ \defeq \{s-1 : s \in S, s > 0\} \, , \\
    \Delta(d,S)^{-\e_i} \ &\cong \
    \Delta(d-1,S^-) \qquad\text{where } S^- \defeq \{ s : s \in S, s < d-1 \} \, .
    \end{aligned}
\end{equation}

These faces will be helpful in determining the edges of $\Delta(d,S)$.  For
two sets $A,B \subseteq [d]$, we denote the \Def{symmetric difference} of $A$
and $B$ by $A \triangle B \defeq (A \cup B) \setminus (A \cap B)$. For two
points $\p,\q \in \R^d$, we write $[\p,\q]$ for the segment joining $\p$ to
$\q$.

\begin{thm}\label{thm:all_edges}
    Let $S = \{ 0 \le s_1 < \dots < s_k \le d\}$ and $A,B \in \binom{[d]}{S}$
    with $|A|=s_i \le s_j = |B|$. Then $[\e_A, \e_B]$ is an edge of $\Delta(d,
    S)$ if and only if 
    \begin{compactenum}[\rm (i)]
    \item $A \subset B$ and $j  = i+1$, or
    \item $i = j$, $|A \triangle B| = 2$, and $\{s_i - 1, s_i + 1\}
        \not\subset S$.
    \end{compactenum}
\end{thm}
\begin{proof}
    Let $A, B \in \binom{[d]}{S}$. If $i \in A \cap B$, then
    $[\e_A,\e_B]$ is an edge of $\Delta(d,S)$ if and only if
    $[\e_A,\e_B]$ is an edge of
    $\Delta(d,S)^{\e_i}$. By~\eqref{eqn:Delta_facets},
    $\Delta(d,S)^{\e_i} \cong \Delta(d-1,S^+)$ and $[\e_A,\e_B]
    \cong[\e_{A \setminus i},\e_{B \setminus i}]$. Hence we can assume
    $A \cap B = \emptyset$. For $i \in [d] \setminus (A \cup B)$, we
    consider $\Delta(d,S)^{-\e_i}$ and by the same argument we may also
    assume that $A \cup B = [d]$.
    
    If $A = \emptyset$, then $B = [d]$ and $[\e_A,\e_B]$ meets every
    $\Delta(d,k)$ in the relative interior for $0 < k < d$. Hence
    $[\e_A, \e_B]$ is an edge if and only if $S = \{0,d\}$, which
    gives us (i).
    
    If $0 < s_i = |A|$, then let $i \in A$ and $j \in B$. Then $[\e_A,
    \e_B]$ and $[\e_{A'}, \e_{B'}]$ have the same midpoint for $A' = (A
    \setminus i) \cup j$ and $B' = (B \setminus j) \cup i$. Thus
    $[\e_A,\e_B]$ is an edge of $\Delta(d,S)$ if and only if $(A',B')
    = (B,A)$. This is the case precisely when $|A \triangle B| = 2$
    and $A \cap B, A \cup B \not \in \binom{[d]}{S}$.
\end{proof}

Theorem~\ref{thm:all_edges} makes the number of edges readily available.

\begin{cor}\label{cor:num_edges}
    The number of edges of $\Delta(d, S)$ is
    \[
        \sum_{i=1}^{k} \binom{d-s_i}{s_{i+1}-s_i}\binom{d}{s_i} + \sum_j
        \frac{s_j(d-s_j)}{2} \binom{d}{s_j} \, ,
    \]
    where we set $s_{k+1} = 0$ and the second sum is over all $1 \le j \le k$,
    such that $\{s_j - 1, s_j + 1\} \not\subset S$.
\end{cor}

Let us illustrate Theorem~\ref{thm:all_edges} for the classical
examples of $S$-hypersimplices. For $\Cube_d = \Delta(d,[0,d])$ it
states, that the edges are of the form $[\e_A,\e_B]$ for any $A
\subset B \subseteq [d]$ such that $|A| + 1 = |B|$.  For the halfcube
$H_d = \Delta(d,[0,d]\cap 2\Z)$ we infer that there are $d(d-1)
2^{d-3}$ many edges for $d \ge 3$. As for the cross-polytope $
\Delta(d,\{1,d-1\})$, every two vertices are connected by an edge,
except for $\e_{\{i\}}$ and $\e_{[d] \setminus \{i\}}$ for all $i \in
[d]$.

Theorem~\ref{thm:all_edges} states that there are no \emph{long} edges of
$\Delta(d,S)$.  We can make use of this fact to get a canonical decomposition
of $\Delta(d,S)$.  For $\lambda \in \R$, define the hyperplane 
\[
    H(\lambda) \ \defeq \ \{ \x \in \R^d : x_1 + \cdots + x_d = \lambda \} \,
    .
\]
We note the following consequence of Theorem~\ref{thm:all_edges}.

\begin{cor}
    Let $S \subseteq [0,d]$ and $s \in S$. Then $\Delta(d,S) \cap H(s)
    = \Delta(d,s)$.
\end{cor}
\begin{proof}
    Every vertex $\v$ of $\Delta(d,S) \cap H(s)$ is of the form $F \cap H(s)$
    for a unique inclusion-minimal face $F \subseteq \Delta(d,S)$ of dimension
    $\le 1$. If $F$ is an edge, then its endpoints $\e_A,\e_B$ satisfy $|A| <
    s < |B|$ which contradicts Theorem~\ref{thm:all_edges}. Hence $F = \e_C$
    for some $C \subseteq [d]$ with $|C| = s$.
\end{proof}

If $S = \{s_1 < \cdots < s_k\}$ with $k \ge 2$, then we can decompose
\begin{equation}\label{eqn:decomp}
    \Delta(d,S) \ = \ 
    \Delta(d,s_1,s_{2}) \ \cup \
    \Delta(d,s_2,s_{3}) \ \cup \
    \cdots \ \cup \
    \Delta(d,s_{k-1},s_{k}) \, ,
\end{equation}
where we set $\Delta(d,k,l) \defeq \Delta(d,\{k,l\}) = \conv( \Delta(d,k) \cup
\Delta(d,l))$ for $0 \le k < l \le d$.  The polytope $\Delta(d,k,l)$ is the
\Def{Cayley polytope} of $\Delta(d,k)$ and $\Delta(d,l)$. Moreover, for $i <
j$, we see that $\Delta(d,s_i,s_{i+1}) \cap \Delta(d,s_j,s_{j+1}) =
\Delta(d,s_j)$ if $j = i +1$ and $=\emptyset$ otherwise.

Before we determine the facets of $\Delta(d,S)$, we recall some properties of
permutahedra from~\cite{BilleraSarangarajan} that we will also need in
Section~\ref{sec:mono}. A point $\p \in \R^d$ is \Def{decreasing} if $p_1 \ge
p_2 \ge \cdots \ge p_d$. The \Def{permutahedron} associated to $\p$ is the
polytope
\[
    \Pi(\p) \ \defeq \ \conv \bigl(  \sigma\p \defeq (
    p_{\sigma(1)},
    p_{\sigma(2)},
    \dots,
    p_{\sigma(d)}) \; : \; \sigma \text{ permutation of } [d] \bigr) \, .
\]
Unless $p_i = p_j$ for all $i \neq j$, $\Pi(\p)$ is a polytope of dimension
$d-1$ with affine hull given by $H(p_1+ \cdots + p_d)$.

Notice that $\Pi(\p)^{\sigma\u} \ = \ \sigma^{-1}\Pi(\p)^\u$.
Thus, if we want to determine the face $\Pi(\p)^\u$ up to permutation of
coordinates, we can assume that $\u$ is decreasing.  The \Def{Minkowski sum}
of two polytopes $P,Q \subset \R^d$ is the polytope $P + Q = \{\p + \q : p \in
P, q \in Q\}$. 

\begin{prop}\label{prop:perm_Mink}
    Let $\p,\q \in \R^d$ be decreasing. Then 
    \[
        \Pi(\p) + \Pi(\q) \ = \ \Pi(\p+\q) \, .
    \]
\end{prop}
\begin{proof}
    Set $P \defeq \Pi(\p) + \Pi(\q)$. Clearly $\sigma(\p + \q) =
    \sigma\p + \sigma\q$ for all permutations $\sigma$ and therefore
    every vertex of $\Pi(\p + \q)$ is a vertex of $P$. For the converse,
    let $\c$ be such that $P^\c = \{\v\}$ is a vertex. Since $P$ is
    invariant under coordinate permutations, we can assume that $\c$ is
    decreasing. Furthermore $(\Pi(\p) + \Pi(\q))^\c = \Pi(\p)^\c +
    \Pi(\q)^\c$ and it follows that $\v = \p + \q$. Hence, every vertex
    of $P$ is of the form $\sigma (\p + \q)$ for some permutation
    $\sigma$, which completes the proof.
\end{proof}

For $\nu_1 > \nu_2 > \cdots > \nu_r$ and $k_1, k_2, \dots, k_r \in \Z_{> 0}$
such that $k_1 + \cdots + k_r = d$, we set
\[
    (\nu_1^{k_1},\nu_2^{k_2},\dots,\nu_r^{k_r}) \ \defeq \
    (\underbrace{\nu_1,\dots,\nu_1}_{k_1},\underbrace{\nu_2,\dots,\nu_2}_{k_2},\dots,
    \underbrace{\nu_r,\dots,\nu_r}_{k_r}) \, .
\]
For example, the $(d,k)$-hypersimplex is the permutahedron $\Delta(d,k) =
\Pi(1^k,0^{d-k})$. 

The facets of permutahedra were described by
Billera-Sarangarajan~\cite{BilleraSarangarajan}. We recall their
characterization. We write $I^c \defeq [d] \setminus I$ for the complement of
$I \subseteq [d]$.

\begin{thm}[{\cite[Theorem~3.2]{BilleraSarangarajan}}]\label{thm:perm_facets}
    Let $P = \Pi(\nu_1^{k_1},\dots,\nu_r^{k_r})$ and $\c \in \R^d$. Then
    $P^\c$ is a facet if and only if $\c = \alpha \e_I + \beta \e_{I^\c}$ for
    some $\alpha > \beta$ and $\emptyset \neq I \subset [d]$ and $h = |I|$
    satisfies
    \begin{compactenum}[\rm (a)]
        \item $k_1 + 1 \le h \le d- k_r - 1$, or
        \item $h = 1$ if $k_1 < d - 1$, or
        \item $h = d-1$ if $k_r < d - 1$.
    \end{compactenum}
\end{thm}

The theorem shows, for example, that $\Delta(d,k)$ for $1 < k < d-1$ has $2d$
facets with normals given by $\pm \e_1,\dots, \pm \e_d$.

In order to determine the facets of $\Delta(d,S)$, we appeal to the
decomposition~\eqref{eqn:decomp}. Let $S = \{ 0 \le s_1 < s_2 < \cdots < s_k
\le d \}$ be proper so that $\Delta(d,S) \subset \R^d$ is full-dimensional. We
write $\1 \defeq \e_{[d]}$ for the all-ones vector. If $s_1 > 0$, then
$\Delta(d,s_1) = \Delta(d,S)^{-\1}$ is a facet. Likewise, if $s_d < d$, then
$\Delta(d,s_k) = \Delta(d,S)^{\1}$ is a facet. If $F \subset \Delta(d,S)$ is
any other facet, then its vertices cannot have all the same cardinality. If
$s_i \in S$ is the minimal cardinality of a vertex in $F$, then $F \cap
\Delta(d,s_i,s_{i+1})$ is a facet of $\Delta(d,s_i,s_{i+1})$.  Hence, as a
first step, we determine the facets of $\Delta(d,s_i,s_{i+1})$ that are not
equal to $\Delta(d,s_i)$ and $\Delta(d,s_{i+1})$.

\newcommand\oDelta{\overline{\Delta}}%
Let $S = \{k < l\}$ be proper. An easy calculation shows that 
\[
    \Delta(d,k,l) \cap H(\tfrac{k+l}{2}) \ = \ \tfrac{1}{2} ( \Delta(d,k) +
    \Delta(d,l) ) \, .
\]
Moreover, if $F \subset \Delta(d,k,l)$ is a facet, then $F \cap
H(\tfrac{k+l}{2})$ is a facet of the right-hand side and every facet arises
that way. Hence it suffices to determine the facets of $\oDelta(d,k,l) \defeq
\Delta(d,k) + \Delta(d,l)$.  We will need the notion of a join of two
polytopes: If $P,Q \subset \R^d$ are polytopes such that their affine hulls
are \emph{skew}, i.e., non-parallel and disjoint, then $P \ast Q \defeq
\conv(P \cup Q)$ is called the \Def{join} of $P$ and $Q$. Every
$k$-dimensional face of $P \ast Q$ is of the form $F \ast G$ where $F
\subseteq P$ and $G \subseteq G$ are (possibly empty) faces with $\dim F +
\dim G = k-1$.

\begin{prop}\label{prop:facets}
    Let $1 \le k < l < d$. In addition to the facets $\Delta(d,k,l)^\1 =
    \Delta(d,l)$ and $\Delta(d,k,l)^{-\1} = \Delta(d,k)$, there are 
    \[
        \Delta(d,k,l)^{\e_i} \ \cong \ \Delta(d-1,k-1,l-1) \quad \text{ and }
        \quad \Delta(d,k,l)^{-\e_i} \ \cong \ \Delta(d-1,k,l) 
    \]
    for $i=1,\dots,d$.  Every other facet is of the form
    \[
        \Delta(d,k,l)^\c \ \cong \ \Delta(h,k) \ast \Delta(d-h,l-h)
    \]
    where $\c = (l-h) \e_I - (h - k) \e_{I^c}$ for any 
    $\emptyset \neq I \subset [d]$ with $k < h \defeq |I| < l$.
\end{prop}
\begin{proof}
    We first determine the facets of $\oDelta(d,k,l)$.  Using
    Proposition~\ref{prop:perm_Mink}, we see that $\oDelta(d,k,l)$ is the
    permutahedron $\Pi(2^k,1^{l-k},0^{d-l})$.  Theorem~\ref{thm:perm_facets}
    yields that the facet directions of $\oDelta(d,k,l)$ are given $\c =
    \alpha \e_I + \beta \e_{I^c}$ for $\emptyset \neq I \subset [d]$ with $|I|
    = 1$, $|I| = d-1$, or $k < |I| < l$ and $\alpha > \beta$. In particular,
    for every $I$ there is, up to scaling, a unique choice for $\alpha$ and
    $\beta$ so that $\Delta(d,k,l)^\c$ is a facet.

    For $I = \{i\}$ we already observed that $\c = \e_I = \e_i$ yields a facet
    linearly isomorphic to $\Delta(d-1,k-1,l-1)$. Likewise, for $[d]
    \setminus I = \{j\}$, we obtain for $\c = \e_I - \1 = -\e_{j}$ a facet
    that is linearly isomorphic to $\Delta(d-1,k,l)$.

    For $I \subseteq [d]$ with $k < |I| < l$, we observe that $\e_A
    \in \Delta(d,k)^{\e_I}$ if and only if $A \subset I$ and $\e_A \in
    \Delta(d,l)^{\e_I}$ if and only if $I \subset A$. Set $h \defeq
    |I|$ and $\c = (l-h) \e_I - (h - k) \e_{I^c}$. For $A \in
    \binom{[d]}{k}$ we compute
    \[
        \inner{\c,\e_A} \ = \ 
        (l-h) |A \cap I| - (h-k) |A \cap I^c| \ \le \ (l-h) k
    \]
    with equality if and only if $A \subset I$.  For $A \in \binom{[d]}{l}$,
    we compute
    \[
        \inner{\c,\e_A} \ = \ 
        (l-h) |A \cap I| - (h-k) |A \cap I^c| \ \le \ (l-h) h - (h-k) (l-h)
        \ = \ (l-h) k
    \]
    with equality if and only if $I \subset A$. Hence the hyperplane $H = \{
    \x : \inner{\c,\x} = (l-h) k \}$ supports $\Delta(d,k,l)$ in a facet,
    since $H$ also supports a facet of $\oDelta(d, k, l)$. In particular,
    $\Delta(d,k) \cap H \cong \Delta(h,k)$ under the identification $\{ \x :
    x_i = 0 \text{ for } i \not\in I \} \cong \R^h$.  Likewise $\Delta(d,l)
    \cap H \cong \Delta(d-h,l-h)$ under the identification $\{ \x : x_i = 1
    \text{ for } i \in I \} \cong \R^{d-h}$.  This also shows that the given
    subspaces are skew and, since they lie in $H(k)$ and $H(l)$ respectively,
    are disjoint. This shows that $\Delta(d,l,k) \cong \Delta(h,k)
    \ast \Delta(d-h,l-h)$.
\end{proof}

It follows from Proposition~\ref{prop:facets} that $\Delta(d,k,l)$ and
$\Delta(d,l,m)$ for $0 < k < l < m < d$ never have facet normals of type (v)
in common. This gives us the following description of facets of
$S$-hypersimplices; see also~\cite{groetschel}.

\begin{thm}\label{thm:facets}
    Let $S = \{ 0 \le s_1 < \cdots < s_k \le d\}$ be proper. Then
    $\Delta(d,S)$ has the following facets
    \begin{compactenum}[\rm (i)]
    \item $\Delta(d,S)^{\1} = \Delta(d,s_k)$ provided $s_k < d$;
    \item $\Delta(d,S)^{-\1} = \Delta(d,s_1)$ provided $0 < s_1$;
    \item $\Delta(d,S)^{\e_i} \cong \Delta(d-1,S^+)$ for $i =
      1,\dots,d$ provided $S^+$ is proper;
    \item $\Delta(d,S)^{-\e_i} \cong \Delta(d-1,S^-)$ for $i = 1,\dots,d$
        provided $S^-$ is proper;
    \item $\Delta(d,S)^{\u_I} \cong \Delta(h, h - s_i) \ast
      \Delta(d-h,s_{i+1}-h)$ where $I \subset [d]$ with $s_i < |I|
      \eqdef h < s_{i+1}$ for some $0 < i < k$ and $\u_I \defeq (s_{i+1}-h)
      \e_I - (h - s_i) \e_{I^c}$.
    \end{compactenum}
\end{thm}
\begin{proof}
    By decomposition~\eqref{eqn:decomp}, every facet $F$ of $\Delta(d,S)$
    determines a facet of $\Delta(d, s_i, s_{i+1})$ for some $1 \le i < k$ and
    $F$ is decomposed by this collection of facets. By examining the possible
    facet normals of $\Delta(d, s_i, s_{i+1})$, the statement readily follows.
\end{proof}

If $S = [0,d]$, then Theorem~\ref{thm:facets} gives us that $\Cube_d$ has exactly
$2d$ facets in the coordinate directions $\pm \e_i$ for $i=1,\dots,d$. The
facets are again cubes as $[0,d]^\pm = [0,d-1]$.
The $d$-dimensional crosspolytope $\Diamond_d  \cong \Delta(d,\{1,d-1\})$ has
$2^d$ facets.  The two facets of type (i), (ii), and those of type (iii) and
(iv) are simplices. As for type (v) this is a join of two simplices and thus
also a simplex.

The description of combinatorial type of each facet also leads to the number
of $k$-dimensional faces for $0 \le k < d$; cf.~\cite{Ritter-master}.

\renewcommand\S{\mathcal{S}}%
\section{Pulling triangulations}\label{sec:pull}

A \Def{subdivision} $\S$ of a $d$-dimensional polytope $P \subset \R^d$ is a
collection $\S = \{P_1,\dots,P_m\}$ of $d$-polytopes such that $P = P_1 \cup
\cdots \cup P_m$ and $P_i \cap P_j$ is a face of $P_i$ and $P_j$ for all $1
\le i < j \le m$. If all polytopes $P_i$ are simplices, then $\S$ is called a
\Def{triangulation}.  Triangulations are the method-of-choice for various
computations on polytopes including volume, lattice point counting, or, more
generally, computing valuations; see~\cite{Triang}.

\newcommand\Pull{\mathsf{Pull}}%
A powerful method for computing a triangulation is the so-called \emph{pulling
triangulation}. Let $P$ be a \mbox{$d$-polytope} and $\v \in V(P)$ a vertex. Let
$F_1,\dots,F_m$ be the facets of $P$ not containing $\v$. A key insight is
that the collection of polytopes
\[
    P_i \ \defeq \ \v \ast F_i \  \defeq \ \conv( F_i \cup \{\v\} ) \qquad
    \text{ for } i =
    1,\dots, m
\]
constitutes a subdivision of $P$. This idea can be extended to obtain
triangulations. Let $\preceq$ be a partial order on the vertex set $V(P)$ such
that every nonempty face $F \subseteq P$ has a unique minimal element with
respect to $\preceq$. We denote the minimal vertex of $F$ by $\v_F$.  The
\Def{pulling triangulation} $\Pull_\preceq(P)$ of $P$ is recursively defined as
follows. If $P$ is a simplex, then $\Pull_\preceq(P) = \{ P \}$. Otherwise, we
define
\begin{equation}\label{eqn:pull}
    \Pull_\preceq(P) \ = \ \bigcup_{F} \ \v_P \ast \Pull_\preceq(F) \, ,
\end{equation}
where the union is over all facets $F \subset P$ that do not contain
$\v_P$ and where $\v_P \ast \Pull_\preceq(F) \defeq \{ \v_P \ast Q : Q
\in \Pull_\preceq(F) \}$.

For the cube $\Cube_d$, or more generally the class of \emph{compressed}
polytopes~\cite{sullivant}, it can be shown that every simplex $S$ in a
pulling triangulation of $\Cube_d$ has the same volume $\frac{1}{d!}$. Thus, every
pulling triangulation has exactly $d!$ many simplices, independent of the
chosen order $\preceq$.

Recall that the halfcube is the $S$-hypersimplex $H_d = \Delta(d,[0,d]
\cap 2\Z)$. For $d \ge 5$ it is not true that the simplices in a pulling
triangulation of $H_d$ all have the same volume. The main result of this
section is that still the number of simplices in a pulling triangulation is
independent of the choice of $\preceq$.

\begin{thm}\label{thm:H_pull}
    Every pulling triangulation of $H_d$ has the same number of simplices. The
    number  of simplices $t(d) \defeq |\Pull_\preceq(H_d)|$ is given by
    \[
        t(d) \ = \ \sum_{l=3}^{d} \frac{d!}{l!} \left(2^{l-1} - l \right) \, .
    \]
\end{thm}

The proof of Theorem~\ref{thm:H_pull} is in two parts. We first show that the
number of simplices of $\Pull_\preceq(H_d)$ is independent of $\preceq$. This
yields a recurrence relation on $t(d)$. In the second part we review 
the construction of $\Pull_\preceq(H_d)$ from the perspective of choosing
facets, which yields a combinatorial interpretation for $t(d)$ and which then
verifies the stated expression.

From Theorem~\ref{thm:facets} we infer the following description of facets of
$H_d$ for $d \ge 3$: For every $i = 1,\dots,d$ we have
\begin{align*}
    H_d^{-\e_i} \ &= \ H_d \cap \{ \x : x_i = 0 \} \ \cong \ H_{d-1} \, ,\\
    H_d^{\e_i} \ &= \ H_d \cap \{ \x : x_i = 1 \} \ \cong \ H_{d-1} \, ,
\end{align*}
where the last isomorphism is realized by reflection in a hyperplane $\{\x :
x_j = \frac 1 2 \}$ for $j \neq i$. The remaining facets of $H_d$ are provided
by Theorem~\ref{thm:facets}(v) and, in case $d$ is odd, by (i): For  
$B \subseteq [d]$ with $|B|$ odd and $\u_B = \e_B - \e_{B^c}$, we have
\[
    H^{\u_B}_d \ = \ H_d \cap \{ \x : \inner{ \e_B, \x} - \inner{ \e_{B^c},
    \x} = |B|-1 \} \ \cong \ \Delta_{d-1} \, .
\]

\begin{prop}\label{prop:H_pull}
    The number $t(d)$ of simplices in a pulling triangulation of $H_d$
    satisfies
    \[
        t(d) \ = \ d \cdot t(d-1) + 2^{d-1} - d
    \]
    for $d \ge 4$ and $t(d) = 1$ for $d \le 3$.
\end{prop}
\begin{proof}
    We prove the result by induction on $d$. For $d=1,2,3$, we note that $H_d$
    is itself a simplex and thus there is nothing to prove.

    For $d \ge 4$, let $A \subseteq [d]$ be an even subset such that $\e_A \in
    \ZO^d$ is the minimal vertex of $P$ with respect to $\preceq$. By the
    discussion preceeding the proposition, the facets not containing $\e_A$
    are $H_d^{\e_i} \cong H_{d-1}$ for $i \not\in A$, $H_d^{-\e_i} \cong
    H_{d-1}$ for $i \in A$, and $H_d^{\u_B} \cong \Delta_{d-1}$ for 
    \[
        B \in \mathcal{B} \ \defeq \ \{ B \subseteq [d] : |B| \text{ odd, } 
    |A \triangle B| > 1 \} \, .
    \]
    Note that $|\mathcal{B}| = 2^{d-1}-d$.  Thus it follows
    from~\eqref{eqn:pull} that 
    \begin{align*}
    t(d) \ &= \ |\Pull_\preceq(H_d)| \ = \ 
    \sum_{i\in A} |\Pull_\preceq(H_d^{-\e_i})| + 
    \sum_{i\not\in A} |\Pull_\preceq(H_d^{\e_i})| + 
    \sum_{B \in \mathcal{B}} |\Pull_\preceq(H_d^{\u_B})| \\
    &= \ d \cdot t(d-1) + 2^{d-1} - d \, ,
    \end{align*}
    where the last equality follows by induction.
\end{proof}

Let $P \subset \R^d$ be a full-dimensional polytope with suitable partial
order $\preceq$ on $V(P)$. Every simplex in $\Pull_{\preceq}(P)$ corresponds
to a chain of faces 
\begin{equation}\label{eqn:pull_simp}
    P \ = \ G_0  \ \supset \ G_1 \ \supset \ G_2 \ \supset \ \cdots \ \supset
    \ G_k
\end{equation}
such that $\dim G_i = d-i$ and $G_k$ is a simplex of dimension
$d-k$. The corresponding simplex is then given by $\v_{G_0} \ast
\v_{G_1} \ast \cdots \ast G_k$.  If $P$ is a simple polytope with
facets $F_1,\dots,F_m$, then any such chain of faces is given by an
ordered sequence of distinct indices $h_1, h_2, \dots, h_k$ such that
\[
    G_i \ = \ F_{h_1} \cap F_{h_2} \cap \cdots \cap F_{h_i} 
\]
for all $i=0,\dots,k$. 

For the $d$-dimensional cube $\Cube_d$, the facets can be described by  
$(i,\delta) \in [d] \times \ZO$ so that 
\[
    K_{i}^\delta \ \defeq \ \Cube_d \cap \{ x_i = \delta \} \ \cong \ \Cube_{d-1} \, .
\]
The only faces of $\Cube_d$ that are simplices have dimensions $\le 1$ and thus
simplices in $\Pull_{\preceq}(\Cube_d)$ correspond to sequences $(i_1,\delta_1),
\dots, (i_{d-1},\delta_{d-1}) \in [d] \times \ZO$ with $i_s \neq i_t$ for
$s \neq t$. Thus, if we choose $i_d$ such that $\{i_1,\dots,i_{d-1},i_d\} =
[d]$, then every simplex of $\Pull_{\preceq}(\Cube_d)$ determines a permutation
$\sigma = i_1i_2\cdots i_d$ of $[d]$. 

Observe that for any vertex $\v \in \Cube_d$ and $i \in [d]$, we have that $\v \in
K_i^0$ or $\v \in K_i^1$. This means that for any permutation $\sigma =
i_1i_2\cdots i_d$ of $[d]$ there are $\delta_1,\delta_2,\dots,\delta_{d-1} \in
\ZO$ such that $(i_1,\delta_1),\dots,(i_{d-1},\delta_{d-1})$ come from a simplex
in $\Pull_{\preceq}(\Cube_d)$. This shows that $|\Pull_{\preceq}(\Cube_d)| = d!$
independent of the order $\preceq$.

We call a sequence $\tau = i_1i_2\dots i_k$ with $i_1,\dots,i_k \in [d]$ a
\Def{partial permutation} if $i_s \neq i_t$ for $s \neq t$. We simply write
$[d] \setminus \tau$ for $[d] \setminus \{ i_1, \dots, i_k \}$. The following
Proposition completes the proof of Theorem~\ref{thm:H_pull}.

\begin{prop}
    For any suitable partial order $\preceq$, the simplices of
    $\Pull_\preceq(H_d)$ for $d \ge 3$ are in bijection to pairs $(\tau, B)$
    where $\tau$ is a partial permutation of $[d]$ and $B \subseteq [d]
    \setminus \tau$ is a non-singleton subset of odd cardinality.
\end{prop}
\begin{proof}
    Since $H_3$ is a simplex and the only admissible pair $(\tau,B)$ is given
    by the empty partial permutation and $B = [3]$, we assume $d \ge 4$.
    For $i=1,\dots,d$ and $\delta \in \ZO$, let
    \[
        F_{i}^\delta \ \defeq \ H_d \cap \{ x_i = \delta \} \ \cong \ H_{d-1}
    \]
    be the halfcube facets of $H_d$. The halfcube $H_d$ for $d \ge 4$ is not a
    simple polytope. However, it follows from Theorem~\ref{thm:facets} that
    the faces of $H_d$ are halfcubes or simplices. If $G \subset H_d$ is a
    face linearly isomorphic to a halfcube of dimension $d-k \ge 4$, then $G$
    is a simple face in the sense that $G$ is precisely the intersection of
    $k$ halfcube facets. Every chain of faces~\eqref{eqn:pull_simp}
    corresponds to some $(i_1,\delta_1),\dots,(i_{k-1},\delta_{k-1}) \in [d]
    \times \ZO$ such that $G_{k-1} = F_{i_1}^{\delta_1} \cap \cdots \cap
    F_{i_{k-1}}^{\delta_{k-1}}$ is isomorphic to $H_{d-k+1}$ and $G_k$ is a
    simplex facet of $G_{k-1}$ not containing $\v_{G_{k-1}}$.  This gives rise
    to a unique partial permutation $\tau = i_1i_2 \dots i_{k-1}$. To see that
    any such partial permutation can arise, we observe that again $V(H_d)
    \subset F_i^0 \cup F_i^1$ for all $i=1,\dots,d$. We can identify $G_{k-1}$
    with $H_{d-k+1}$ embedded in $\{ \x : x_{i_1} = \cdots = x_{i_{k-1}} = 0
    \}$ and $v_{G_{k-1}} = \0$. Now any simplex facet of $H_{d-k+1}$
    corresponds to an odd-cardinality subset $B \subset [d] \setminus \tau$
    with $|B| \neq 1$.
\end{proof}

\section{Monotone paths}\label{sec:mono}

Let $P \subset \R^d$ be a polytope and $\ell : \R^d \to \R$ a linear
function.  An \Def{$\ell$-monotone path} of $P$
is a sequence of vertices $W = \v_1,\v_2,\dots,\v_k$ such that
$[\v_{i},\v_{i+1}]$ is an edge of $P$ for $i=1,\dots,k-1$ and 
\[
    \min \ell(P) \ = \ \ell(\v_1) \ < \ \ell(\v_2) \ < \ \cdots \ < \
    \ell(\v_k) \ = \ \max \ell(P) \, .
\]
More generally, a collection of faces $F_1, F_2, \dots, F_k$  of $P$ is an
\Def{induced subdivision} of the segment $\ell(P)$ if $F_1^{-\ell}$ and
$F_k^{\ell}$ is a face of $P^{-\ell}$ and $P^{\ell}$, respectively, and 
\[
    F_i^{\ell} \ = \ F_{i+1}^{-\ell}
\]
for $i=1,\dots,k-1$. If $\ell$ is \Def{generic}, that is, if $\ell$ is not
constant on edges of $P$, then the minimum/maximum of $\ell$ on every
nonempty face $F$ is attained at a unique vertex. In this case $F^{\pm
\ell}_i$ is a vertex for all $i$ and a induced subdivision is called a
\Def{cellular string}.
An induced subdivision $F'_1,\dots,F'_h$ is a refinement if for every $1 \le
i \le k$, there are $1 \le s < t \le h$ such that $F'_s,\dots,F'_t$ is a
induced subdivision of $\ell(F_i)$. The collection of all induced
subdivisions of $\ell(P)$ is partially ordered by refinement  and is called
the \Def{Baues poset} of $(P,\ell)$.  The minimal elements in the Baues
poset are exactly the $\ell$-monotone paths. Monotone paths are quintessential in
the study of simplex-type algorithms in linear programming but they are also
studied in topology in connection with iterated loop spaces;
see~\cite{BKS,Reiner}. For the linear function $\card(x) = x_1 + \cdots +
x_d$, Corollary~\ref{cor:num_edges} readily yields the $\card$-monotone
paths of $\Delta(d,S)$.

\begin{cor}
    Let $S = \{ s_1 < s_2 < \cdots < s_k \}$ be proper. The $\card$-monotone paths
    correspond to sequences $A_1 \subset A_2 \subset \cdots \subset A_k$ with
    $|A_i| = s_i$ for all $i=1,\dots,k$.
\end{cor}

A $\ell$-monotone path $W$ is \Def{coherent} if $W$ is a monotone path with
respect to the \emph{shadow-vertex algorithm};
see~\cite{Borgwardt,KleeKleinschmidt}. That is, if there is
linear function $h_W : \R^d \to \R$ such that under the projection $\pi :
\R^d \to \R^2$ given by $\pi(\x) = (\ell(\x),h_W(\x))$, the path $W$ is mapped
to one of the two paths in the boundary of the polygon $\pi(P)$.
Figure~\ref{fig:incoh} shows that in general coherent paths constitute
a proper subset of all $\ell$-monotone paths and it is interesting to
determine for which pairs $(P,\ell)$ all $\ell$-monotone paths are coherent;
see, for example, the recent paper~\cite{EJLC}. The $S$-hypersimplices
with the linear function $\card(x)$ are examples of this.

\begin{figure}[h]
	\begin{tikzpicture}[scale=.75]
    \draw [line cap=round, line width=1.5pt] (5.5,3)-- (7,5);
	
	\draw[->] (0,0) -- (8,0) node[right] {\Large $\ell$};
	
	\draw[shift={(1,0)}] (0pt,2pt) -- (0pt,-2pt) node[below] {\large $t_0$};
	\draw[shift={(2.5,0)}] (0pt,2pt) -- (0pt,-2pt) node[below] {\large $t_1$};
	\draw[shift={(5.5,0)}] (0pt,2pt) -- (0pt,-2pt) node[below] {\large $t_2$};
	\draw[shift={(7,0)}] (0pt,2pt) -- (0pt,-2pt) node[below] {\large $t_3$};

	\draw [line cap=round, line width=1.5pt, color=red] (1,5)-- (2.5,3.03) -- (5.53,3.03) -- (7,1.06);
	\draw [line cap=round, line width=1.5pt, color=green] (1,1)-- (2.5,2.97) -- (5.50,2.97) -- (6.97,1);
	\draw [line cap=round, line width=1.5pt, color=blue] (1,5)-- (7,5);

	\draw [line cap=round, line width=1.5pt] (1,1)-- (7,1) -- (7,5) (1,1)-- (1,5);
	
	\draw [line cap=round, line width=1.5pt] (10,2.5)-- (14,2.5) -- (12.8,4) -- (11.2,4) -- (10,2.5);
	
    \begin{scriptsize}
      \draw[color=black] (12,3.2) node {\large $\sum_\ell(P)$};
      \draw[color=black] (0.2,3) node {\Large $P$};
      \draw[color=blue] (3,5.5) node {\large $W_1$};
      \draw[color=green] (2.2,1.8) node {\large $W_2$};
      
      \draw [fill=blue] (10,2.5) circle (3.5pt);
      \draw[color=blue] (9.7,2.0) node {\large $W_1$};
      \draw [fill=green] (12.8,4) circle (3.5pt);
      \draw[color=green] (13.3,4.4) node {\large $W_2$};
	\end{scriptsize}
	
	\end{tikzpicture}
    \caption{Left: Top view of triangular prism $P$ and linear function
    $\ell$. Three $\ell$-monotone paths (in red, green, and blue) but the red
    path is not coherent. Right: Monotone path polytope $\Sigma_\ell(P)$.}
    \label{fig:incoh}
\end{figure}
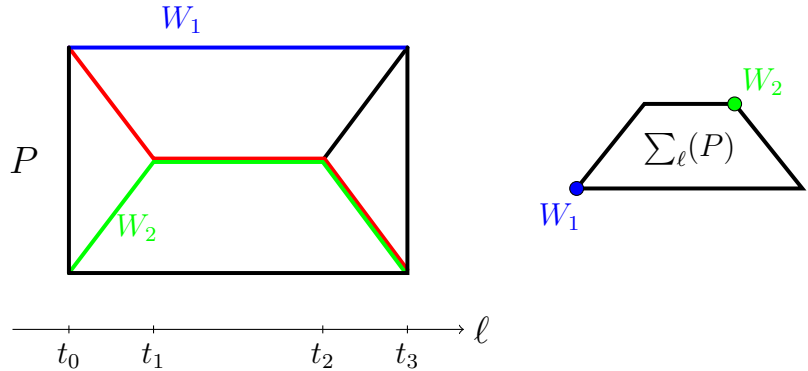

\begin{prop}
    Let $S \subseteq [0,d]$ be proper. Then all $\card$-monotone path of
    $\Delta(d,S)$ are coherent.
\end{prop}
\begin{proof}
    Let $A_1 \subset A_2 \subset \cdots \subset A_k$ be a $\card$-monotone path.
    For the linear function 
    \[
        h(\x) \ \defeq \ \inner{ 
        \1_{A_1} + 
        \cdots +
        \1_{A_k}, \x}
    \]
    it is easy to see that $h(\1_B)$ with $B \in \binom{[d]}{S}$ is maximal
    if and only if $B \in \{A_1,\dots,A_k\}$.
\end{proof}

The \Def{monotone path polytope} $\Sigma_\ell(P)$ is a convex polytope of
dimension $\dim P - 1$ whose face lattice is isomorphic to the poset of
coherent subdivisions. The construction is a special case of fiber polytopes
of Billera and Sturmfels~\cite{BS}. Let $\ell(P) = [a,b] \subset \R$.  A
\Def{section} of $(P,\ell)$ is a continuous function $\gamma : [a,b] \to P$ such
that $\ell(\gamma(t)) = t$ for all $a \le t \le b$. Following~\cite{BS}, the
monotone path polytope is defined as
\[
    \Sigma_\ell(P) \ = \ \conv \left\{ \frac{1}{b-a}\int_P \gamma \, d\x :
    \gamma \text{ section} \right\} \, .
\]

We now determine the monotone path polytopes of $\Delta(d,S)$ with respect to
the natural linear function $\card(x) = x_1 + \cdots + x_d$. Let us first
observe that for $S \subset [d-1]$ the $\card$-monotone paths of $\Delta(d,S)$
and $\Delta(d,S \cup \{0,d\})$ are in bijection. Clearly every
$\card$-monotone path of $\Delta(d,S \cup \{0,d\})$ restricts to a
$\card$-monotone path of $\Delta(d,S)$. Conversely, if $A_1 \subset \cdots
\subset A_k$ corresponds to a $\card$-monotone path, then $\emptyset \eqdef A_0
\subset A_1 \subset \cdots \subset A_k \subset A_{k+1} = [d]$ is the unique
extension to a $\card$-monotone path of $\Delta(d,S \cup \{0,d\})$.

\begin{thm}\
    Let $S = \{ 0 = s_0 \le s_1 < s_2 < \cdots < s_{k-1} < s_{k} = d \}$ be
    proper. Then
    \[
        \tfrac{1}{2} \1 + d \cdot \Sigma_{\card}(\Delta(d,S)) \ = \
        \Pi(k^{s_1-s_0},(k-1)^{s_2-s_1},\dots,1^{s_{k}-s_{k-1}}) \, .
    \]
\end{thm}
\begin{proof}
    Let $P \subset \R^d$ be a polytope with vertex set $V$ and let $\ell$ be a
    linear function. Let $\ell(V) = \{ a = t_0 < t_1 < \cdots < t_k = b \}$.
    We write $P_i \defeq P \cap \ell^{-1}(t_i)$ for $0 \le i \le k$. Theorem~1.5
    of~\cite{BS} together with the fact that 
    \[
        P \cap \ell^{-1}\left(\frac{t_i + t_{i+1}}{2}\right) \ = \
        \frac{1}{2}(P_i + P_{i+1})
    \]
    for $0 \le i < m$ yields that
    \[
        (b-a) \Sigma_\ell(P) \ = \ \tfrac{1}{2} P_0 + \sum_{i=1}^{k-1} P_i +
        \tfrac{1}{2} P_k \, .
    \]
    If $P = \Delta(d,S)$ and $\ell(\x) = \card(\x)$, then $P_i = \Delta(d,s_i)$
    for $0 \le i \le k$. In particular, $P_0 = \{\0\}$ and $P_k = \{\1\}$.
    Therefore
    \[
        \tfrac{1}{2} \1 + d \cdot \Sigma_{\card}(\Delta(d,S)) \ = \
        \sum_{i=0}^k \Delta(d,s_i) \, .
    \]
    Since $\Delta(d,s_i) = \Pi(1^{s_i},0^{d-s_i})$ we conclude from
    Proposition~\ref{prop:perm_Mink} that the above sum is the permutahedron
    $\Pi(\p)$ for
    \[
        \p = (1^{s_0},0^{d-s_0}) + \cdots + (1^{s_k},0^{d-s_k}) \, .
    \]
    This finishes the argument.
\end{proof}

\section{Further questions}\label{sec:misc}

\subsection*{Volumes and Gr\"obner bases}
Laplace and later Stanley~\cite{StanleyHyp} showed that the volume of
$\Delta(d,i,i+1)$ is $\frac{A(d,i)}{d!}$ where $A(d,i)$ counts the number of
permutations $\sigma$ of $[d]$ with $i$ \Def{descents}, that is, the number of
$1 \le i < d$ such that $\sigma(i) > \sigma(i+1)$; see
also~\cite{LP,SanyalStump}. This implies that $d! \, \vol
\Delta(d,[k,l])$ is the number of permutations of $[d]$ with descent number in
$[k,l] = \{k,k+1,\dots,l\}$ for any $k < l$. It would be very interesting to
know if $\vol \Delta(d,S)$ has a combinatorial interpretation for all $S$. In
light of~\eqref{eqn:decomp} it would be sufficient to determine $\vol
\Delta(d,k,l)$ for $l - k > 1$.
   
For $0 \le k < d$, the hypersimplices $\Delta(d,k,k+1) \cong \Delta(d,k+1)$
are \emph{alcoved polytopes} in the sense of
Lam--Postnikov~\cite{LP} and hence come with a canonical square-free
and unimodular triangulation. This is reflected by the fact that the
associated toric ideals have quadratic and square-free Gr\"obner bases with
respect to the reverse-lexicographic term order.

For general $k < l$, the polytopes $\Delta(d,k,l)$ are not alcoved anymore. It
would be interesting if $\Delta(d,k,l)$ has a unimodular triangulation or
square-free Gr\"obner basis.

\subsection{Extension complexity}
\newcommand\ext{\operatorname{ext}}
An \Def{extension} of a polytope $P$ is a polytope $Q$ together with a
surjective linear projection $Q \to P$. The \Def{extension complexity}
$\ext(P)$ of $P$ is the minimal number of facets of an extension of $P$. This
is a parameter that is of interest in combinatorial
optimization~\cite{Kaibel11a}. It was shown in~\cite{ExtHypSimp} that
$\ext(\Delta(d,k,k+1)) = 2d$ for $1 \le k \le d-2$.

A realization of the join of two polytopes $P, Q \subset \R^d$ is given by $P
\ast Q = \conv( (P \times \0 \times 0) \cup (\0 \times Q \times 1) )$. If $P$
and $Q$ has $m$ and $n$ facets, respectively, then $P*Q$ has $m+n$ facets.
Balas' union bound~\cite{Balas} is the observation that $P \ast Q \to P \cup
Q$ and hence $\ext(P \cup Q) \le \ext(P) + \ext(Q)$. Iterating the join over
the pieces of the decomposition~\ref{eqn:decomp} shows the following.

\begin{prop}\label{prop:ext}
    If $S \subseteq [0,d]$ is proper, then
    \[
        \ext(\Delta(d,S)) \ \le \ 2d \, (|S|-1) \, .
    \]
\end{prop}

This is a nontrivial bound as the number of facets of $\Delta(d,S)$ is at
least $2 + 2d + \sum_{r \not\in S} \binom{d}{r}$. To illustrate, note that the
number of facets of the halfcube $H_d$ for $d \ge 5$ is $2d + 2^{d-1}$ whereas
the bounded afforded by Proposition~\ref{prop:ext} is $\le d^2$. Carr and
Konjevod~\cite{CarrKonjevod} gave an extension of $H_d$ of size linear in $d$.
It would be interesting to know lower bounds on the extension complexity of
$\Delta(d,S)$, maybe using the approach via rectangular covering;
c.f.~\cite{ExtHypSimp}.

\bibliographystyle{siam} \bibliography{bibliography.bib}

\begin{thebibliography}{10}

\bibitem{Balas}
{\sc E.~Balas}, {\em Disjunctive programming: properties of the convex hull of
  feasible points}, Discrete Appl. Math., 89 (1998), pp.~3--44.

\bibitem{BKS}
{\sc L.~J. Billera, M.~M. Kapranov, and B.~Sturmfels}, {\em Cellular strings on
  polytopes}, Proc. Amer. Math. Soc., 122 (1994), pp.~549--555.

\bibitem{BilleraSarangarajan}
{\sc L.~J. Billera and A.~Sarangarajan}, {\em The combinatorics of permutation
  polytopes}, in Formal power series and algebraic combinatorics ({N}ew
  {B}runswick, {NJ}, 1994), vol.~24 of DIMACS Ser. Discrete Math. Theoret.
  Comput. Sci., Amer. Math. Soc., Providence, RI, 1996, pp.~1--23.

\bibitem{BS}
{\sc L.~J. Billera and B.~Sturmfels}, {\em Fiber polytopes}, Ann. of Math. (2),
  135 (1992), pp.~527--549.

\bibitem{Borgwardt}
{\sc K.-H. Borgwardt}, {\em The simplex method}, vol.~1 of Algorithms and
  Combinatorics: Study and Research Texts, Springer-Verlag, Berlin, 1987.
\newblock A probabilistic analysis.

\bibitem{CarrKonjevod}
{\sc R.~D. Carr and G.~Konjevod}, {\em Polyhedral Combinatorics}, Springer New
  York, New York, NY, 2005, pp.~2--1--2--46.

\bibitem{coxeter}
{\sc H.~S.~M. Coxeter}, {\em Regular polytopes}, Dover Publications, Inc., New
  York, third~ed., 1973.

\bibitem{Triang}
{\sc J.~A. De~Loera, J.~Rambau, and F.~Santos}, {\em Triangulations}, vol.~25
  of Algorithms and Computation in Mathematics, Springer-Verlag, Berlin, 2010.

\bibitem{EJLC}
{\sc R.~Edman, P.~Jiradilok, G.~Liu, and T.~McConville}, {\em Zonotopes whose
  cellular strings are all coherent}, preprint arXiv:1801.09140,  (2018).

\bibitem{ErmelWalter}
{\sc D.~Ermel and M.~Walter}, {\em Parity polytopes and binarization}, Discrete
  Applied Mathematics,  (2018).

\bibitem{Gosset}
{\sc T.~{Gosset}}, {\em {On the regular and semi-regular figures in space of
  $n$ dimensions.}}
\newblock {Messenger (2) 29, 43-48 (1899).}, 1899.

\bibitem{ExtHypSimp}
{\sc F.~Grande, A.~Padrol, and R.~Sanyal}, {\em Extension complexity and
  realization spaces of hypersimplices}, Discrete Comput. Geom., 59 (2018),
  pp.~621--642.
\newblock
  \href{https://doi.org/10.1007/s00454-017-9925-4}{doi:10.1007/s00454-017-9925-4}.

\bibitem{Green}
{\sc R.~M. Green}, {\em Homology representations arising from the half cube},
  Adv. Math., 222 (2009), pp.~216--239.

\bibitem{Green2}
\leavevmode\vrule height 2pt depth -1.6pt width 23pt, {\em Homology
  representations arising from the half cube, {II}}, J. Combin. Theory Ser. A,
  117 (2010), pp.~1037--1048.

\bibitem{groetschel}
{\sc M.~Gr\"{o}tschel}, {\em Cardinality homogeneous set systems, cycles in
  matroids, and associated polytopes}, in The sharpest cut, MPS/SIAM Ser.
  Optim., SIAM, Philadelphia, PA, 2004, pp.~99--120.

\bibitem{Kaibel11a}
{\sc V.~Kaibel}, {\em Extended formulations in combinatorial optimization}.
\newblock Optima 85, 2011.
\newblock 14 pages.

\bibitem{KleeKleinschmidt}
{\sc V.~Klee and P.~Kleinschmidt}, {\em Geometry of the {G}ass-{S}aaty
  parametric cost {LP} algorithm}, Discrete Comput. Geom., 5 (1990),
  pp.~13--26.

\bibitem{LP}
{\sc T.~Lam and A.~Postnikov}, {\em Alcoved polytopes. {I}}, Discrete Comput.
  Geom., 38 (2007), pp.~453--478.

\bibitem{MS}
{\sc D.~Maclagan and B.~Sturmfels}, {\em Introduction to tropical geometry},
  vol.~161 of Graduate Studies in Mathematics, American Mathematical Society,
  Providence, RI, 2015.

\bibitem{Reiner}
{\sc V.~Reiner}, {\em The generalized {B}aues problem}, in New perspectives in
  algebraic combinatorics ({B}erkeley, {CA}, 1996--97), vol.~38 of Math. Sci.
  Res. Inst. Publ., Cambridge Univ. Press, Cambridge, 1999, pp.~293--336.

\bibitem{Ritter-master}
{\sc J.~Ritter}, {\em The polytopes of cardinality homogeneous set systems},
  {M}asterthesis, FU Berlin, 2016.
\newblock iii+44 pages.

\bibitem{SSS}
{\sc R.~Sanyal, F.~Sottile, and B.~Sturmfels}, {\em Orbitopes}, Mathematika, 57
  (2011), pp.~275--314.

\bibitem{SanyalStump}
{\sc R.~Sanyal and C.~Stump}, {\em Lipschitz polytopes of posets and
  permutation statistics}, J. Combin. Theory Ser. A, 158 (2018), pp.~605--620.
\newblock
  \href{https://doi.org/10.1016/j.jcta.2018.04.006}{doi:10.1016/j.jcta.2018.04.006}.

\bibitem{StanleyHyp}
{\sc R.~P. Stanley}, {\em Eulerian partitions of a unit hypercube}, in Higher
  Combinatorics, M.~Aigner, ed., D. Reidel Publishing Co., Dordrecht-Boston,
  Mass., 1977.

\bibitem{sullivant}
{\sc S.~Sullivant}, {\em Compressed polytopes and statistical disclosure
  limitation}, Tohoku Math. J. (2), 58 (2006), pp.~433--445.

\bibitem{Yannakakis}
{\sc M.~Yannakakis}, {\em Expressing combinatorial optimization problems by
  linear programs}, J. Comput. System Sci., 43 (1991), pp.~441--466.

\end{thebibliography}

\end{document}